\documentclass{amsart}
\usepackage{amsmath,amssymb}
\usepackage{hyperref,amssymb,tikz}
\usepackage[margin=1.5cm]{geometry}
\newcommand{\Sym}{\mathop{\mathrm{Sym}}}
\newcommand{\Alt}{\mathop{\mathrm{Alt}}}

\renewcommand{\wr}{\mathop{\mathrm{wr}}}
\newtheorem{theorem}{Theorem}[section]
\newtheorem{corollary}[theorem]{Corollary}

\newtheorem{lemma}[theorem]{Lemma}
\newtheorem{remark}[theorem]{Remark}

\numberwithin{equation}{section}
\author[D.~Bubboloni]{Daniela Bubboloni}
\address{Daniela Bubboloni, Dipartimento di Matematica e Informatica,\newline
University of Firenze, \newline Viale Morgagni  $67$/a, 50134 Firenze, Italy}
\email{daniela.bubboloni@unifi.it}

\author[C. E. Praeger]{Cheryl E. Praeger}
\address{Cheryl E. Praeger, Centre for Mathematics of Symmetry and Computation,\newline
School of Physics, Mathematics, and Computing,\newline
The University of Western Australia,\newline
 Crawley, WA 6009, Australia}\email{Cheryl.Praeger@uwa.edu.au}

\author[P. Spiga]{Pablo Spiga}
\address{Pablo Spiga,
Dipartimento di Matematica e Applicazioni,\newline
 University of Milano-Bicocca,\newline Via Cozzi 55, 20125 Milano, Italy
}
\email{pablo.spiga@unimib.it}

\thanks{The first author is partially supported by GNSAGA of INdAM.
 The second author is supported by Australian Research Council Discovery Project Grant
 DP130100106}
\subjclass[2010]{20B30, 20F05}
\keywords{symmetric groups; conjugacy classes; normal coverings; partitions.}

\begin{document}
\title[Linear bounds ]{ Linear bounds for the normal covering number of the symmetric and alternating groups}

\begin{abstract}
The normal covering number $\gamma(G)$ of a finite, non-cyclic group $G$ is the minimum number
of proper subgroups  such that each element of $G$  lies in some conjugate
of one of these subgroups. We find lower bounds linear in $n$ for $\gamma(S_n)$, when $n$ is even, and  for $\gamma(A_n)$, when $n$ is odd.
\end{abstract}

\maketitle
\section{Introduction}\label{sec:0}
Let $G$ be a non-cyclic finite group. We write $\gamma(G)$ for the
smallest number of conjugacy classes of proper subgroups of $G$ needed to cover it and we call $\gamma(G)$ the {\em normal covering number} of $G$.
In other words, $\gamma(G)$ is the least $s$ for which there exist subgroups $H_1,\ldots,H_s<G$
such that
$$G =
\bigcup_{g\in G}\bigcup_{i=1}^s H_i^g.$$
We say that $\Sigma :=\{H_i^g \mid 1 \le i \le s,\, g \in G\}$ is a {\em normal covering}
for $G$ and the subgroups in $\Sigma$ are called its {\em components}. Moreover, the $H_i$ are called the {\em basic components} of $\Sigma$
 and  $\delta := \{H_1 , \ldots , H_s \}$ a {\em basic set} for $G$ which generates $\Sigma$. If $s=\gamma(G)$, we say that  $\Sigma$ is a {\em minimal normal covering} for $G$ and that  $\delta$ is a {\em minimal basic set}.

The problem of determining the normal covering number of a finite group arises from Galois theory and, broadly speaking, it is related to the investigation of integer polynomials having a root modulo $p$, for every prime number $p$, see~\cite{BBH}, ~\cite{BS} and ~\cite{RS}. Because of the ubiquity of the alternating and the symmetric groups as Galois groups, determining good lower and upper bounds for $\gamma(G)$ when $G$ is one of those groups is a natural and important problem. From now on, let $G$ be equal to $\Sym(n)$ with $n\geq 3,$ or to $\Alt(n)$ with $n \geq 4$. Note that for smaller values of $n$, the numbers $\gamma(\Sym(n))$ and  $\gamma(\Alt(n))$ are not defined, since normal coverings do not exist for cyclic groups.

The first two authors of this paper found~\cite[Theorems~A and~B]{JCTA} that $$a\varphi(n) \le \gamma(G) \leq bn,$$ where $\varphi(n)$ is the Euler totient
function and $a$, $b$ are positive real constants depending on whether $G$ is alternating
or symmetric and whether $n$ is even or odd.
The bounds above on the normal covering number open the question on whether $\gamma(G)$ grows linearly with the degree $n$ or with $\varphi(n)$. In~\cite{BPS}, the authors gave an answer to this question, showing that
\begin{equation}\label{lb}
cn\leq \gamma(G)\leq \frac{2}{3}n,
\end{equation}
for some explicitly computable constant $c\in(0,1/2]$. For instance, when $n$ is even and $n\ge 792\, 000$, $\gamma(\Sym(n))\ge 0.025\cdot n$, see~\cite[Remark~$6.5$]{BPS}.

That result was reached with a combination of two techniques: the first, group-theoretic related to an invariant defined by Britnell and Mar\'oti in~\cite{BM}, the second, number theoretic related to the asymptotic enumeration of certain partitions satisfying some coprimality conditions~\cite{BFS}. (Incidentally, the group invariant introduced by Britnell and Mar\'oti turns out to be an excellent tool for the determining the normal covering number of linear groups~\cite{BM}.)

Comparing the linear lower bound in \eqref{lb} with computer computations, one quickly realises that there is a lot of room for improvement. Somehow this is reflected by the fact that ultimately the proofs in~\cite{BPS} use number theoretic results on partitions, which do not entirely capture the group structure of the symmetric and alternating groups.

In this paper we use the classification in \cite{GMPS1,GMPS} of finite primitive permutation groups containing a
permutation with at most four cycles to improve the results in~\cite{BPS}. We obtain a new and largely better linear lower bound for the normal covering number of the symmetric and alternating groups.
\smallskip

Modern discrete mathematics  uses computers extensively for preliminary observations, for proving ``small" cases, for checking or for constructing  ``working conjectures''. Based on extensive computer computations, we proposed the following conjecture, which appeared as Problem 18.23 in the Kourovka notebook~\cite[18.23]{kourovka}.
\smallskip

\noindent {\it Conjecture\/}: Write $n=p_1^{\alpha_1}\cdots p_r^{\alpha_r}$ for primes  $p_1<\cdots<p_r$ and positive integers $\alpha_1,\ldots,\alpha_r$, $r\in\mathbb{N}$. Then

\begin{align}\label{dag}
 \gamma(\Sym(n))&=\left\{
\begin{array}{lcl}
\frac{n}{2}\left(1-\frac{1}{p_1}\right)&&\textrm{if }r=1\,\,\textrm{and}\,\,\alpha_1=1\\
\frac{n}{2}\left(1-\frac{1}{p_1}\right)+1&&\textrm{if }r=1\,\,\textrm{and}\,\,\alpha_1\geq 2\\
\frac{n}{2}\left(1-\frac{1}{p_1}\right)\left(1-\frac{1}{p_2}\right)+1&&\textrm{if }r=2\,\,\textrm{and}\,\,\alpha_1+\alpha_2=2\\
\frac{n}{2}\left(1-\frac{1}{p_1}\right)\left(1-\frac{1}{p_2}\right)+2&&\textrm{if }r\geq 2\,\,\textrm{and}\,\,\alpha_1+\cdots+\alpha_r\geq 3\\
\end{array}
\right.
\end{align}
This is the strongest form of our conjecture. We were really interested in knowing whether it was true for sufficiently large values of $n$. The equalities for $r\leq 2$, which include the first three cases above, were proved  for $n$ odd, see~\cite{JCTA}. Moreover, we showed in~\cite{IJGT}, that the conjectured value for $\gamma(\Sym(n))$ is an upper bound, by constructing a normal covering for $\Sym(n)$ with the conjectured number of conjugacy classes of maximal subgroups,  and we gave further evidence for the truth of the conjecture in other cases. However, we emphasise again that the main evidence for the conjectured value of $\gamma(\Sym(n))$ came from our computer calculations.

\smallskip

Recently Attila Mar\'oti showed, in private correspondence with the authors, that this conjecture is incorrect and we present his counterexample in Section \ref{symmetric-part}, with his permission. Mar\'oti asked, in the same correspondence, whether perhaps  $\gamma(\Sym(n))>\frac{n}{\pi^2}.$ In this paper we prove a better lower bound than that suggested by Mar\'oti for the case where $n$ is even, and we prove similar lower bounds for $\gamma(\Alt(n))$ for $n$ odd.

\begin{theorem}\label{t:summary}
Suppose that $n\ge 20$, and either $G=\Sym(n)$ with $n$ even or $G=\Alt(n)$ with $n$ odd. Then
$$
\gamma(G)\ge
\frac{n}{2}\left(1-\sqrt{1-8\zeta_2(n)}\right)-\sqrt{n}, \quad \mbox{where}\quad
\zeta_2(n):=\frac{1}{12}\prod_{\substack{p\mid n\\p\,\mathrm{ prime}}}\left(1-\frac{1}{p^2}\right).
$$
In particular,
$$
\gamma(G)\ge
\frac{n}{2}\left(1-\sqrt{1- 4/\pi^2}\right)-\frac{\sqrt{17}}{2}n^{3/4}\ \approx\ 0.114411\, n -\frac{\sqrt{17}}{2}n^{3/4}.
$$
\end{theorem}

\begin{remark}\label{rem}{\rm
\begin{enumerate}
\item[(a)] Theorem~\ref{t:summary} is proved as part of Corollary~\ref{cor:nr1} which, in turn, follows from Theorem~\ref{gamma-sym-even}, our main result about $\gamma(G)$.

\item[(b)]
Noting that $\pi^{-2}\approx 0.101321$, we see that the lower bound in Theorem~\ref{t:summary} is asymptotically better than the one suggested by Attila Mar\'oti. We comment on this further after the proof of Corollary~\ref{cor:nr1}.

\item[(c)]
The cases not covered by Theorem~\ref{t:summary}, namely $\Sym(n)$ for $n$ odd and $\Alt(n)$ for $n$ even, would require considerably more delicate
tools.  Our approach to proving  Theorem~\ref{t:summary} has involved a rigorous analysis of the numbers and types of maximal subgroups required to contain, between them, conjugates of all permutations having exactly three cycles of lengths, say, $a,b,c$ such that $n=a+b+c$ and  $\gcd(a,b,c)=1$. The group $\Alt(n)$, for $n$ even, contains no such permutations, while in $\Sym(n)$, for $n$ odd, all such permutations  lie in $\Alt(n)$ and so can be covered by the single component $\Alt(n)$ in a basic set for $\Sym(n)$. If this problem were to be tackled by an extension of our approach then a similar, very technical, analysis would be required of subgroups containing permutations with four cycles of lengths $a,b,c,d$ such that $\gcd(a,b,c,d)=1$.
We discuss some of the difficulties of dealing with such permutations with four cycles at the end of Section 7.
\end{enumerate}
}
\end{remark}

\section{Preliminaries and Mar\'oti's construction}\label{symmetric-part}

We denote by $\mathbb{N}$ the set of positive integers. Given $n,k\in\mathbb{N}$,  a $k$-{\em partition} of $n$ is an unordered $k$-tuple $\mathfrak{p}=[x_1,\dots,x_k]$ where, for every
$j\in \{1,\dots,k\}$, $x_j\in\mathbb{N}$ and $n=\sum_{j=1}^{k}x_j.$
The numbers $x_1,\ldots, x_k$ are called the {\em terms} of $\mathfrak{p}$.
The set of $k$-partitions of $n$ is denoted by $\mathfrak{P}_k(n)$. We stress that, by definition, in $\mathfrak{p}\in \mathfrak{P}_k(n)$ repetitions and reordering of terms are allowed.
Note that $\mathfrak{P}_k(n)=\varnothing$ for $k>n.$  A {\em partition} of $n$ is an element of
$\mathfrak{P}(n):=\bigcup_{k\in \mathbb{N}} \mathfrak{P}_k(n)$ and $\mathfrak{P}(n)$  is  called the set of partitions of $n$.

We denote by $\Sym(n)$ the symmetric group of degree  $n\geq 3$ and by $\Alt(n)$ the alternating group of degree $n\geq 4$, both considered in their natural action on $\Omega=\{1,\dots,n\}$. Partitions are very important in dealing with the normal coverings of $\Sym(n)$ and $\Alt(n)$. Indeed,
let $\sigma\in \Sym(n)$ and let $k$ be the number of orbits  of $\langle\sigma\rangle$ on $\Omega$. The
 {\em type}  of  $\sigma$ is the $k$-partition of $n$ given by  the unordered list $\mathfrak{p}(\sigma)=[x_1,...,x_k]\in \mathfrak{P}(n)$ of the sizes $x_i$  of those orbits.
The map
$$\mathfrak{p}:\Sym(n)\rightarrow \mathfrak{P}(n)$$ is
surjective, that is, each partition of $n$ may be viewed as the type of
some permutation. Moreover, $\mathfrak{p}(\Alt(n))\subsetneqq
\mathfrak{P}(n)$.

If $H\leq \Sym(n)$ and $ \mathfrak{p}\in\mathfrak{P}(n)$, we say that $H$ {\em covers} $ \mathfrak{p}$ (or, with abuse of language, that $ \mathfrak{p}$ belongs to $H$)  if $ \mathfrak{p}\in \mathfrak{p}(H),$ that is, $H$ contains a permutation of type $\mathfrak{p}.$

Two permutations are conjugate in  $\Sym(n)$ if and only if they
have the same type.  Thus  $\delta = \{H_1 , \ldots , H_s \}$ is a  basic set for $\Sym(n)$ if and only if every partition of $n$ is covered by some $H_i\in \delta.$
In particular, two permutations of $\Alt(n)$ which are conjugate in $\Alt(n)$ have the same type. Thus if $\delta = \{H_1 , \ldots , H_s \}$ is a  basic set for $\Alt(n)$,  then every partition of $n$ belonging to $\Alt(n)$ is covered by some $H_i\in \delta.$
Recall that, the set of permutations of $\Alt(n)$ with the same type sometimes splits into two conjugacy classes of $\Alt(n)$. This fact will be not relevant for our investigation.

Let $G\in \{\Sym(n),\, \Alt(n) \}$. In order to compute $\gamma(G)$, we can always replace a normal covering of $G$
 by one with the same
number of basic components in which each component is a maximal
subgroup of $G$. For this reason we can assume (and we will) that each
component is a maximal subgroup of $G$, that is to say, in
its action on $\Omega$, it is a maximal intransitive, or
imprimitive, or  primitive subgroup of $G$.

Let $x\in\mathbb{N}$, with  $1\leq x<n/2.$  Consider the subgroup $$P_x:= \Sym(\{1,\ldots,x\})\times \Sym(\{x+1,\ldots,n\})$$ of $\Sym(n)$ fixing the partition $\{\{1,\dots,x\},\{x+1,\dots, n\}\}$ of $\Omega$.  Then the set of maximal subgroups of $G$ which are
intransitive is given, up to conjugacy, by
$$\mathcal{P}:=\{P_x\cap G\mid 1\leq x<n/2\}.$$

Moreover, as is well known, the set of  maximal subgroups of $G$ which are imprimitive is given, up to conjugacy, by the subgroups in the set  $$
\mathcal{W}:=\{\ [\Sym(b)\wr \Sym(n/b)]\cap G \ : 2\leq b\leq n/2,\, b\mid n \}.
$$

Throughout the paper, $\omega(n)$ denotes the number of distinct prime divisors of $n\in \mathbb{N}$, for $n\geq 2$.

We present now the construction by Mar\'oti. Consider $\Sym(n)$, for $n\geq 3$ not a prime, and take all the subgroups in the following sets:
\begin{enumerate}
\item\label{cov1}  $P_k$, where $1\leq k\leq n/3$;
\item\label{cov3}  $P_k$, where $1\leq k<n/2$ and $k$ is coprime to $n$;
\item\label{cov2}  $\Sym(p) \mathrm{wr} \Sym(n/p)$, for the prime divisors $p$ of $n$.
\end{enumerate}
It is easily seen that the above subgroups form a basic set for $\Sym(n)$. Indeed, the subgroups in~\eqref{cov1} cover all the $\ell$-partitions of $n$ with $\ell\geq 3$. The subgroups in ~\eqref{cov2} cover the $2$-partitions with terms divisible by some prime $p$ dividing $n$.  The subgroups in~\eqref{cov3} cover all the $2$-partitions of $n$ with terms coprime to $n$. Thus every partition of $n$ is covered and so these subgroups form a basic set for $\Sym(n)$. Hence,
\begin{equation}\label{eq:a1}
\gamma(\Sym(n)) \leq \frac{n}{3}  + \frac{\varphi(n)}{2}+ \omega(n).
\end{equation}

Now we specialise this construction to particular kinds of natural numbers. Every $n\in \mathbb{N}$ with $n>1$ may be written in the form
\begin{equation}\label{n}
n=p_1^{\alpha_1}\cdots p_r^{\alpha_r}
\end{equation}
 for primes  $p_1<\cdots<p_r$ and positive integers $\alpha_1,\ldots,\alpha_r$, $r\in\mathbb{N}$.
Consider the set $U$ of natural numbers $n$ of this form, with $r\geq 2$ and such that the smallest prime divisor $p_1$ satisfies
\begin{equation}\label{eq:a2}\frac{n}{2}\left(1-\frac{1}{p_1}\right)\left(1-\frac{1}{p_2}\right) + 2 > \frac{n}{3} + \frac{n}{7}.
\end{equation}
Using elementary arguments one may show that \eqref{eq:a2} holds whenever $p_1\geq 43.$
Now for $r\geq 2$ and a prime $p\geq 43$, let $N(r, p)$ be the product of $r$ consecutive primes, the first of which is $p$. Then $N(r, p)\in U$ and
we have $N(r,p)> p^r \geq 43^r>14r=14\,\omega(N(r,p))$, so that
\begin{equation}\label{eq:a3}\omega(N(r,p))  < \frac{N(r,p)}{14}.
\end{equation}
Moreover, since $\displaystyle{\lim_{r\to\infty}\frac{\varphi(N(r,p))}{N(r,p)}= 0},$ there exists  $r^*$  such that, for every $r\geq r^*$,
\begin{equation}\label{eq:a4}
\frac{\varphi(N(r,p))}{2} < \frac{N(r,p)}{14}.
\end{equation}
From~\eqref{eq:a1},~\eqref{eq:a2},~\eqref{eq:a3} and~\eqref{eq:a4}, for $n := N(r,p)=p_1 p_2\dots p_r$, with $r\geq r^*$ and $p_1=p\geq43$, we then have
$$
\gamma(\Sym(n)) < \frac{n}{3} + \frac{n}{14} + \frac{n}{14} <\frac{n}{2}\left(1-\frac{1}{p_1}\right)\left(1-\frac{1}{p_2}\right) + 2.
$$
This contradicts  the conjectured formula \eqref{dag} for infinitely many primes $p$ and integers $r$. Note that this argument involves only odd numbers $n$. It leaves open the question of whether the conjecture \eqref{dag} could hold for even degrees $n$.

\section{Number of partitions and coprime partitions}\label{sec:1}

In this section we collect some formulas about the number  of partitions and coprime partitions and state an easy lemma about a particular kind of $3$-partition which we will meet later.

Let $ p_k(n):=|\mathfrak{P}_k(n)|.$
Obviously  we have $p_1(n)=1$ for all $n\in\mathbb{N}$, $p_k(n)=0$ for all $k>n$, and $p_k(k)=1$ for all $k\in \mathbb{N}$.
Moreover, it is well known that (see ~\cite[page~$81$]{Andrews})
\begin{equation}\label{2-parti}
p_2(n)=\left\lfloor \frac{n}{2}\right\rfloor
\end{equation}
and that
\begin{equation}\label{3-parti}
p_3(n)=\frac{(n-1)(n-2)}{12}+
\frac{1}{2}
\left\lfloor\frac{n-1}{2}\right\rfloor+\varepsilon_n\quad \textrm{with}\quad
\varepsilon_n=
\begin{cases}
0&\textrm{if }3\nmid n,\\
\frac{1}{3}&\textrm{if }3\mid n.
\end{cases}
\end{equation}
 The set of {\it coprime} $k$-partitions of $n$ is defined by
\begin{equation} \label{kcopset}
 \mathfrak{P}_k(n)':=\left\{[x_1,\dots,x_k]\in \mathfrak{P}_k(n) \mid \gcd (x_1,\dots,x_k)=1\right\}.
 \end{equation}
 We also put
\begin{equation} \label{kcopsize}
p_k(n)':=|\mathfrak{P}_k(n)'|.
\end{equation}
Clearly $p_k(n)'=0$ for all $k>n$ and $p_k(k)'=1$ for all $k\in \mathbb{N}$.
Formulas for $p_2(n)'$ and $p_3(n)'$ are given in the following lemma.

 \begin{lemma}\label{l:1}
 Let $n\in\mathbb{ N}.$ Then the following facts hold:
 \begin{itemize}

 \item[i)] $p_2(n)'=\left\lceil\frac{\varphi(n)}{2}\right\rceil$. In particular, for $n			\geq 3$,  $p_2(n)'=\frac{\varphi(n)}{2}$.

\item[ii)] For $n\geq 4$,
 $$p_3(n)'=\frac{n^2}{12} \prod_{\substack{p\mid n\\p\, \mathrm{ prime}}}\left(1-\frac{1}{p^2}\right) .$$
If $n\geq 3$, then $p_3(n)' > \frac{n^2}{2\pi^2}.$
\end{itemize}
\end{lemma}

\begin{proof} The formulas for $p_2(n)'$  and $p_3(n)'$ are proved in~\cite[Theorem 1.1 and Theorem 2.2]{bach}.
Thus, we just need to show the inequality in ii).
For $n\geq 4$, we have
\begin{eqnarray*}p_3(n)'&=&\frac{n^2}{12} \prod_{\substack{p\mid n\\p\, \mathrm{ prime}}}\left(1-\frac{1}{p^2}\right)> \frac{n^2}{12} \prod_{p\, \mathrm{ prime}}\left(1-\frac{1}{p^2}\right)\\
&=& \frac{n^2}{12} \zeta(2)^{-1}=\frac{n^2}{2\pi^2},
\end{eqnarray*}
where $\zeta$ denotes the Riemann zeta function (see, for example \cite[Corollary 8.19 and Theorem A.6]{NZM} for the last equality).
For $n=3$, the inequality can be checked directly.
\end{proof}

\begin{lemma}\label{family}
Let $q$ be a prime power, $d, d_1, d_2 \in \mathbb{N}$ be such that $d=d_1+d_2$ and $\gcd(d_1,d_2)=1$. Then
$$
\left[\frac{q^{d_1}-1}{q-1},\,\frac{q^{d_2}-1}{q-1},\,\frac{(q^{d_1}-1)(q^{d_2}-1)}{q-1}\right]\in \mathfrak{P}_3\left(\frac{q^d-1}{q-1}\right)'.
$$
\end{lemma}

\begin{proof}
Let $a\in \mathbb{N}$ be a divisor of all the terms in $\mathfrak{p}=[\frac{q^{d_1}-1}{q-1},\,\frac{q^{d_2}-1}{q-1},\,\frac{(q^{d_1}-1)(q^{d_2}-1)}{q-1}]$. Then $a\mid \gcd(q^{d_1}-1,q^{d_2}-1)=q^{\gcd(d_1,d_2)}-1=q-1$ and, since $\gcd(\frac{q^{d_i}-1}{q-1},q-1)\mid \gcd(d_i,q-1)$ for $i=1,2$, we get $a\mid \gcd(d_1,d_2)=1.$
\end{proof}

 \section{Partitions with clusters}\label{clusters}

Let $n,k\in \mathbb{N}$ with $n\geq 2$, and let $x\in\mathbb{N}$ with  $1\leq x\leq n-1.$ We say that a $k$-partition $\mathfrak{p}=[x_1,\ldots,x_k]\in \mathfrak{P}_k(n)$ has an {\it $x$-cluster} if the sum of some of its terms is equal to $x$, that is, $x=x_{i_1}+x_{i_2}+\cdots+x_{i_\ell}$ for some $\ell\in\mathbb{N}$ and  $1\le i_1<i_2<\cdots<i_\ell\le k$. Obviously a $k$-partition of $n$ has an $x$-cluster if and only if it has an $(n-x)$-cluster. We are interested in the set
\begin{equation}\label{cluster}
\mathfrak{P}_k(n,x):=\{\mathfrak{p}\in \mathfrak{P}_k(n) \mid\mathfrak{p} \hbox{ has an } x\hbox{-cluster}\}.
\end{equation}
We call $\mathfrak{P}_k(n,x)$ the set of $k$-partitions of $n$ with an $x$-cluster.

We set $p_k(n,x):=|\mathfrak{P}_k(n,x)|$. Obviously, $p_1(n,x)=0$ and $p_2(n,x)=1$. Thus  we are interested in $p_k(n,x)$ only  for $k\geq 3.$ Moreover, since $\mathfrak{P}_k(n,x)=\mathfrak{P}_k(n,n-x)$, we have $p_k(n,x)=p_k(n,n-x)$; therefore for computing $p_k(n,x)$ we may restrict to the case $1\leq x\leq n/2.$
For the purposes of this paper, it will be sufficient to deal with the more restrictive condition $1\leq x< n/2.$
Indeed, when $1\leq x< n/2,$ $\mathfrak{P}_k(n,x)$ coincides with the set of $k$-partitions of $n$ covered by $P_x$. This fact justifies our interest in the sets $\mathfrak{P}_k(n,x).$

\begin{lemma}\label{l:2}  Let $n, x,k\in\mathbb{N}$, with $n\geq 2$ and  $1\leq x\leq n-1$. Then the following facts hold:
\begin{itemize}
\item[i)] \begin{equation}\label{a}
p_{k}(n,x)\leq \sum_{i=1}^{k-1}p_i(x)p_{k-i}(n-x),
\end{equation}
with equality when $k\in\{1,2\}.$
\item[ii)] For $k=3,$ equality  in \eqref{a} holds if and only if $n$ is odd, or $n$ is even and $x\neq\frac{n}{2}
  $, or $(n,x)=(2,1).$ Moreover,
\begin{equation}\label{p3x}
p_3(n,x)=\left\{ \begin{array}{ll}
\left\lfloor \frac{n}{4}\right\rfloor  & {\text{\rm if}} \  n\geq 4\  {\text{\rm is even and }} x=n/2,\\
& \\
\left\lfloor \frac{n-x}{2}\right\rfloor + \left\lfloor \frac{x}{2}\right\rfloor &  {\text{\rm otherwise}.}
\end{array}
\right.
\end{equation}
In particular, $p_3(n,x)\leq \frac{n}{2}.$

 \end{itemize}
\end{lemma}

\begin{proof}

i) For each $i\in\mathbb{N}$ with $i\leq k-1$, consider the set
$$
\mathfrak{P}_k(n,x)_i:= \{\mathfrak{p}\in \mathfrak{P}_k(n) \mid \hbox{ the sum of } i  \hbox{ terms of } \mathfrak{p}  \hbox{ equals } x\}.
$$
If a $k$-partition is covered by $ \mathfrak{P}_k(n,x),$ then, for some $i\in \{1,\ldots,k-1\}$, the sum of $i$ of its terms must be equal to $x$. Therefore, we have
$$
\mathfrak{P}_k(n,x)= \bigcup_{i=1}^{k-1}\mathfrak{P}_k(n,x)_i
$$
and thus $p_{k}(n,x)\leq \sum_{i=1}^{k-1}|\mathfrak{P}_k(n,x)_i|.$  It remains to show that, for each $i\in\{1,\ldots,k-1\}$,
\begin{equation}\label{T}
|\mathfrak{P}_k(n,x)_i|\leq p_i(x)p_{k-i}(n-x).
\end{equation}
Consider the map
$$
f:\mathfrak{P}_i(x)\times \mathfrak{P}_{k-i}(n-x)\rightarrow \mathfrak{P}_k(n,x)_i
$$
which sends $([x_1,\dots, x_i],[x_1',\dots,x_{k-i}'])\in \mathfrak{P}_i(x)\times \mathfrak{P}_{k-i}(n-x)$ to the
$k$-partition $[x_1,\dots, x_i, x_1',\dots,x_{k-i}']$.
Then $f$ is surjective. Indeed, let $\mathfrak{q}\in \mathfrak{P}_k(n,x)_i$ and label its terms by $x_1,\dots, x_k$ so that $\sum_{j=1}^ix_j=x$. Thus $\mathfrak{p}:=[x_1,\dots, x_i]\in \mathfrak{P}_i(x)$, $\mathfrak{p}':=[x_{i+1},\dots, x_k]\in \mathfrak{P}_i(n-x)$ and $f(\mathfrak{p},\mathfrak{p'})=\mathfrak{q}.$
Then,~\eqref{T} follows immediately.
Note now that $p_{1}(n,x)=0$ and hence, when $k=1$, the inequality in~\eqref{a} is an equality, because the sum is taken over an empty set of indices. When $k=2$, we have $p_{2}(n,x)=|\{[x,n-x]\}|=1$ and  also $\sum_{i=1}^{1}p_i(x)p_{k-i}(n-x)=p_1(x)p_{1}(n-x)=1.$

ii) For $k=3$, we have
 \begin{equation}\label{U}
  \mathfrak{P}_3(n,x)= \mathfrak{P}_3(n,x)_1\cup  \mathfrak{P}_3(n,x)_2=\mathfrak{P}_3(n,x)_1\cup  \mathfrak{P}_3(n,n-x)_1.
  \end{equation}
Suppose that either $n$ is odd, or $n$ is even with $x\neq\frac{n}{2}.$ Then $x\neq n-x$ and, since $x$ and
$n-x$ cannot appear together in a $3$-partition of $n$, we have  $\mathfrak{P}_3(n,x)_1\cap  \mathfrak{P}_3(n,n-x)_1
=\varnothing.$ Thus
$$
p_{3}(n,x)=|\mathfrak{P}_3(n,x)_1|+|\mathfrak{P}_3(n,n-x)_1|=p_2(n-x)+p_2(x),
$$
and equality holds in \eqref{a}. Also, if $(n,x)=(2,1)$, then in \eqref{a} both members are equal to $0$ and equality holds.
In all  these cases, from \eqref{2-parti}, we  get
$$
p_{3}(n,x)=\left\lfloor \frac{n-x}{2}\right\rfloor + \left\lfloor \frac{x}{2}\right\rfloor,
$$
and hence~\eqref{p3x} holds.
Finally, suppose that  $n$ is even, $n\geq 4$, and $x=\frac{n}{2}.$
We need to show that $p_3(n,n/2) = \lfloor n/4\rfloor$, and that the inequality in \eqref{a} is strict.
Since trivially $\mathfrak{P}_3(n,n/2)_1=  \mathfrak{P}_3(n,n-n/2)_1$, by \eqref{U} and \eqref{2-parti}, we obtain $p_3(n,n/2)=p_2(n/2)=\left\lfloor \frac{n}{4}\right\rfloor.$ Moreover
$$
\sum_{i=1}^{2}p_i(n/2)p_{3-i}(n/2)=2p_2(n/2)\neq p_2(n/2)
$$
because, since $n/2\geq 2$, we have $p_2(n/2)\neq 0$. The final assertion of ii), namely  $p_3(n,x)\leq \frac{n}{2}$,
follows immediately from~\eqref{p3x}.
%
\end{proof}

\section{Intersections of $3$-partitions with clusters}\label{intersections}
\begin{lemma}\label{l:3}
Let $1\leq x<y<n/2$. Then
$$
\mathfrak{P}_3(n,x)\cap \mathfrak{P}_3(n,y)=\{[x,y,n-x-y],[x,y-x,n-y]\}
$$
has size $2$.
\end{lemma}

\begin{proof}
Clearly $\{[x,y,n-x-y],[x,y-x,n-y]\}\subseteq \mathfrak{P}_3(n,x)\cap \mathfrak{P}_3(n,y)$.
On the other hand, if $\mathfrak{p}\in \mathfrak{P}_3(n,x)\cap \mathfrak{P}_3(n,y)$, then
either $x$ or $n-x$ is a term of $\mathfrak{p}$ and either $y$ or $n-y$ is a term of $\mathfrak{p}$.
However, $n-x > n-y > n/2 > y > x$, so if $n-x$ were a term of $\mathfrak{p}$ then the remaining
two terms would sum to $x$ and hence neither of them could be $y$ or $n-y$. Thus $x$ must be a
term of $\mathfrak{p}$, and so $\mathfrak{p}\in \{[x,y,n-x-y],[x,y-x,n-y]\}$. Finally, if  $[x,y,n-x-y]
=[x,y-x,n-y]$, then either $y=y-x$ or $y=n-y$,  contradicting $x\ge 1$ or $y< n/2$, respectively.
\end{proof}

\begin{lemma}\label{l:4}
Let $1\leq x<y<z<n/2$. Then
\[\mathfrak{P}_3(n,x)\cap \mathfrak{P}_3(n,y)\cap \mathfrak{P}_3(n,z)=
\begin{cases}
\varnothing&\textrm{if }z\notin\{x+y, n-x-y\},\\
\{[x,y,n-x-y]\}&\textrm{if }z\in \{x+y, n-x-y\}.
\end{cases}
\]
\end{lemma}

\begin{proof}
If $z\in\{x+y, n-x-y\}$, then $[x,y,n-x-y]\in \mathfrak{P}_3(n,x)\cap \mathfrak{P}_3(n,y)\cap \mathfrak{P}_3(n,z)$. On the other hand, assume that there exists $\mathfrak{p}\in \mathfrak{P}_3(n,x)\cap \mathfrak{P}_3(n,y)\cap \mathfrak{P}_3(n,z)$.Then, by Lemma \ref{l:3}, we have
\begin{align}\label{eq:neq}\nonumber
\mathfrak{p}\in &(\mathfrak{P}_3(n,x)\cap \mathfrak{P}_3(n,y))\cap(\mathfrak{P}_3(n,y)\cap \mathfrak{P}_3(n,z))\\
&=\{[x,y,n-x-y],[x,y-x,n-y]\}\cap \{[y,z,n-y-z],[y,z-y,n-z]\}.
\end{align}
Therefore, we have four cases to consider:
\begin{enumerate}
\item\label{hyp1} $\mathfrak{p}=[x,y,n-x-y]=[y,z,n-y-z]$,
\item\label{hyp2} $\mathfrak{p}=[x,y,n-x-y]=[y,z-y,n-z]$,
\item\label{hyp3} $\mathfrak{p}=[x,y-x,n-y]=[y,z,n-y-z]$,
\item\label{hyp4} $\mathfrak{p}=[x,y-x,n-y]=[y,z-y,n-z]$.
\end{enumerate}
If~\eqref{hyp1} holds true, since $x\ne z$, we have $x=n-y-z$ and hence $z=n-x-y$. Moreover, the intersection in~\eqref{eq:neq}
reduces to $\{[x,y,n-x-y]\}$. If~\eqref{hyp2} holds true, since $x<n/2< n-z$, we have $x=z-y$ and hence $z=x+y$. Moreover, the
intersection in~\eqref{eq:neq} reduces again to $\{[x,y,n-x-y]\}$.
We easily see that the cases~\eqref{hyp3} and~\eqref{hyp4} are not possible. In~\eqref{hyp3}, the term $n-y$ does not appear in $[y, z, n-y-z]$
since $n-y > n/2 > z > y$ and $n-y > n-y-z$; in~\eqref{hyp4},  the term $n-y$ does not appear in $[y, z-y, n-z]$
since $n-y > n/2 > z > z-y$, $n-y > n/2 > y$, and $n-y > n-z$.
\end{proof}

\begin{lemma}\label{empty}
Let $1\leq x<y<z<t<n/2$. Then
\[\mathfrak{P}_3(n,x)\cap \mathfrak{P}_3(n,y)\cap \mathfrak{P}_3(n,z)\cap \mathfrak{P}_3(n,t)=\varnothing.
\]

\end{lemma}
\begin{proof}
Assume, by contradiction , that $\mathfrak{P}_3(n,x)\cap \mathfrak{P}_3(n,y)\cap \mathfrak{P}_3(n,z)\cap \mathfrak{P}_3(n,t)\neq\varnothing.$ Then, by Lemma \ref{l:4},  the intersection we are dealing with is just $\{\mathfrak{p}\}$ with $\mathfrak{p}=[x,y,n-x-y]=[z,t,n-z-t].$ Since $x,y,z,t$ are distinct, we obtain a contradiction.
\end{proof}

\section{Coprime $3$-partitions covered by imprimitive or intransitive subgroups}\label{im-int}

We begin by recalling a particular case of~\cite[Lemma 4.2]{BPS}.

\begin{lemma}\label{l:7}
Let $n\in\mathbb{N},$ with $n\geq 3$. Then the number of coprime $3$-partitions of $n$ covered by some imprimitive subgroup of $\Sym(n)$ is at most $2n^{3/2}.$
\end{lemma}
\begin{lemma}\label{l:new} Let $n\in\mathbb{N},$ with $n\geq 3$ and $\ell\in \mathbb{N}$. Then the number of $3$-partitions of $n$ covered by $\ell$ intransitive maximal subgroups is at most $\frac{\ell}{2}(n-\ell+1)$.
\end{lemma}
\begin{proof}  Let $\mathcal{I}:=\{P_{x_i} \mid i\in\{1,\dots,\ell\}\}$ be a set of $\ell$  intransitive maximal subgroups of $\Sym(n)$, so that  $1\leq x_i<n/2\,$ for each $i\in\{1,\dots,\ell\}.$ Put $X:=\{x_1,\dots, x_{\ell}\}.$ Then the number of $3$-partitions of $\Sym(n)$ covered by the subgroups in $\mathcal{I}$ is the cardinality of the set $\bigcup_{i=1}^\ell\mathfrak{P}_3(n,x_i)$, where the sets $\mathfrak{P}_3(n,x_i)$ are defined in \eqref{cluster}. We estimate this by the inclusion-exclusion principle.
Consider the sets
\begin{eqnarray*}
&A:=&\{(x_i,x_j)\in X^2 \mid x_i<x_j\},\\
&B:=&\{(x_i,x_j,x_k)\in X^3 \mid x_i<x_j<x_k ,\, x_k=x_i+x_j \hbox{ or } x_k=n-x_i-x_j\},\\
&C:=&\{(x_i,x_j)\in X^2 \mid x_i<x_j \hbox{ and } x_i+x_j\in X \hbox{ or } n-x_i-x_j\in X \}.
\end{eqnarray*}
and note that $|A|=\frac{\ell(\ell-1)}{2}$ and $|C|\leq  \frac{\ell(\ell-1)}{2}.$

Define now the map $f:C\rightarrow X^3$, given for every $(x_i,x_j)\in C$, by
\begin{equation}\label{f-int}
f(x_i,x_j):=\left\{ \begin{array}{ll}
(x_i,x_j, x_i+x_j)  & {\text{\rm if}} \  x_i+x_j\in X,\\
& \\
(x_i,x_j, n-x_i-x_j)  & {\text{\rm if}} \  n-x_i-x_j\in X.
\end{array}
\right.
\end{equation}
Note that $f$ is well defined since  we cannot have both $x_i+x_j\in X$ and $n-x_i-x_j\in X$, as the numbers in $X$ are less than $n/2$. Moreover,
$f$ is injective and $B\subseteq f(C)$, and hence $|B|\leq |f(C)|= |C|\leq \frac{\ell(\ell-1)}{2}.$
Then, by Lemmas~\ref{l:2}\,ii),~\ref{l:3},~\ref{l:4} and~\ref{empty}, we get
\begin{align*}
|\bigcup_{i=1}^\ell\mathfrak{P}_3(n,x_i)|= &\sum_{i=1}^{\ell}|\mathfrak{P}_3(n,x_i)|-\sum_{x_i<x_j}|\mathfrak{P}_3(n,x_i)\cap \mathfrak{P}_3(n,x_j)|\\&+\sum_{x_i<x_j<x_k}|\mathfrak{P}_3(n,x_i)\cap \mathfrak{P}_3(n,x_j)\cap \mathfrak{P}_3(n,x_k)|\\
\leq& \frac{\ell n}{2}-2|A|+|B|\leq \frac{\ell n}{2}-2\frac{\ell(\ell-1)}{2}+\frac{\ell(\ell-1)}{2}\\
\leq& \frac{\ell n-\ell(\ell-1)}{2}.~\qedhere
\end{align*}
\end{proof}

Extending our approach to deal with $\Sym(n)$ for $n$ odd, and $\Alt(n)$ for $n$ even,
would entail generalising Lemmas~\ref{l:3},~\ref{l:4} and~\ref{empty} for $4$-partitions with clusters, in order to control the intersection of finite families of sets $\mathfrak{P}_4(n,x)$.
This does not seem to be an easy adaptation, and as a consequence it is not clear how to obtain an extension of Lemma~\ref{l:new} to $4$-partitions. In addition to these difficulties, note that no exact formula for $p_4(n)'$ is known.

\section{ Coprime $3$-partitions covered by primitive subgroups }\label{sec:5}

In this section we study  partitions in the set $\mathfrak{P}_3(n)'$ defined in \eqref{kcopset} covered by a primitive proper subgroup, that is, a primitive subgroup of $\Sym(n)$ not containing $\Alt(n)$.

\begin{lemma}\label{-2} Let $n\in \mathbb{N}$, with $n\geq 3$. If there exists a proper primitive group of degree $n$ containing a permutation of type $\mathfrak{p}=[x_1,x_2, x_3]\in \mathfrak{P}_3(n)',$
 then $n$ and $\mathfrak{p}$ are listed in {\rm Table ~\ref{table1}} when $n$ is even, and in {\rm Table ~\ref{table2}} when $n$ is odd.
\end{lemma}

\begin{proof}
The proof follows from \cite[Theorem 1.1]{GMPS} after a careful inspection of Tables~5--7 in~\cite{GMPS1} and Tables~1--5 in~\cite{GMPS} to select the coprime $3$-partitions. The fact that the partitions in the last lines of   {\rm Tables~\ref{table1} and~\ref{table2}} are coprime is guaranteed by Lemma~\ref{family}.
\end{proof}

We emphasise that on lines $1$-$2$ of Table ~\ref{table1}  and lines $1$-$4$ of Table~\ref{table2}, the column headed Types contains a complete list of such partitions.

\begin{table}[!h]\caption{Primitive types: $n$ even}
\begin{tabular}{|c|c|c|}\hline
Degree $n$&Types &Comments\\\hline
\vspace{-7mm}\\
& &\\
$10$& $[1,3,6],\ [1,1,8]$&\\
\vspace{-7mm}\\
& &\\
$22$&$ [1,7,14]$&\\
\vspace{-7mm}\\
& &\\
$2^a$& $[2, 2^{a-1}-1, 2^{a-1}-1]$&$a\geq 2$\\
\vspace{-7mm}\\
& &\\
$\frac{q^d-1}{q-1}$&$\left[\frac{q^{d_1}-1}{q-1},\,\frac{q^{d_2}-1}{q-1},\,\frac{(q^{d_1}-1)(q^{d_2}-1)}{q-1}\right]$&$q$ odd prime power, $d_1,d_2\geq 1$,\\
&& $d=d_1+d_2$ even, $\gcd(d_1,d_2)=1$\\\hline
\end{tabular}
\label{table1}
\end{table}

\begin{table}[!h]\caption{Primitive types: $n$ odd}

\begin{tabular}{|c|c|c|}\hline

Degree&Types&Comments\\\hline
\vspace{-7mm}\\
& &\\
$7$&$[1,3,3],\ [1,2,4]$&\\
\vspace{-7mm}\\
& &\\
$11$&$[1,5,5],\ [2,3,6],\ [1,2,8]$&\\
\vspace{-7mm}\\
& &\\
$17$&$[1,8,8],\ [2,3,12],\ [2,5,10]$&\\
\vspace{-7mm}\\
& &\\
$23$&$[1,11,11],\ [2,7,14],\  [3,5,15]$&\\
& &\\
\vspace{-7mm}\\
$p^a$&$[1,\frac{p^a-1}{2},\frac{p^a-1}{2}]$& $p$ odd, $a\geq 1$\\
& &\\
\vspace{-7mm}\\
$p^2$&$[1,p-1,p(p-1)]$& $p$ odd\\
& &\\
\vspace{-7mm}\\
$\frac{q^d-1}{q-1}$&$\left[\frac{q^{d_1}-1}{q-1},\,\frac{q^{d_2}-1}{q-1},\,\frac{(q^{d_1}-1)(q^{d_2}-1)}{q-1}\right]$& $q$ prime power, $d_1,d_2\geq 1$,\\
&& $d=d_1+d_2,$ $\gcd(d_1,d_2)=1,$\\
&& $q$ odd when $d$ is odd\\\hline

\end{tabular}
\label{table2}
\end{table}
Before proving our main result, in the next section, we need some technical number theoretic
facts.

\begin{lemma}\label{phi} If $d\in\mathbb{N}$ is even, then $\varphi(d)\leq d/2.$
\end{lemma}

\begin{proof}
Let $d=2^kv$ with $k,v\in \mathbb{N}$ and $v$ odd: then $\varphi(d)=\varphi(2^k)\varphi(v)=2^{k-1}\varphi(v)\leq 2^{k-1}v=d/2.$
\end{proof}



We now define, for each $n\in\mathbb{N}$, the set
\begin{equation}\label{Q}
Q(n):=\left\{(q,d)\in \mathbb{N}^2\mid  n=\frac{q^d-1}{q-1}, q \hbox{ is a prime power}, d \geq 2\right\}.
\end{equation}

\begin{lemma}\label{-1} Let $n\in\mathbb{N}$, with $n\geq 3$.
Then the following hold:
\begin{itemize}
\item[i)] $|Q(n)|\le \omega(n-1)$;
\item[ii)]if $(q,d)\in Q(n),$ then $d < \log_q(n)+1$;
\item[iii)] if $(q,2)\in Q(n),$ then $|Q(n)|=1.$
 \end{itemize}
\end{lemma}

\begin{proof} i)  We first observe that if $(p^a,d)\in Q(n)$, for some prime number $p$ and some positive integers $a$ and $d$, with $d\ge 2$, then $p$ divides $n-1,$
because $n-1=p^a\,\frac{p^{a(d-1)}-1}{p^a-1}.$
Let $p$ be a prime number and let $a$ and $b$ be positive integers such that $(p^a,d_a),(p^b,d_b)\in Q(n)$, for some positive integers $d_a$ and $d_b$ with $d_a,d_b\ge 2$. Then
$$
p^{a(d_a-1)}+p^{a(d_a-2)}+\cdots +p^{2a}+p^a=n-1=p^{b(d_b-1)}+p^{b(d_b-2)}+\cdots +p^{2b}+p^b.
$$
From the left hand side we see that $p^a$ is the largest power of $p$ dividing $n-1$, whereas, from the right hand side with see that $p^b$ is the largest power of $p$ dividing $n-1$. Therefore $a=b$ and hence $d_a=d_b$. This proves that, for every prime number $p$, the equation
$$n=\frac{p^{ad}-1}{p^a-1}$$
has at most one solution in integers $a,d$ with $a\ge 1$ and $d\ge 2$. Hence, $|Q(n)|\le \omega(n-1)$.

ii) Let $(q,d)\in Q(n)$. Then  $n=\frac{q^d-1}{q-1}$ and thus $n>q^{d-1}.$  It follows that $d-1\leq \log_q(n)$.

iii) Assume that $(p^a,2)\in Q(n)$, for some prime number $p$ and some positive integer $a$. Then  $Q(n)\neq \varnothing$ and $n-1=p^a$, so that $p$ is the only prime dividing $n-1$. By i), we then have $|Q(n)|\le \omega(n-1)=1$ and therefore $|Q(n)|=1.$
\end{proof}

In fact we do not know whether there could be more than one integer solution to the equation $n=\frac{q^d-1}{q-1}$, when $n$ is even, $q$ is an odd prime power, and $d$ is  even. On the other hand, it was first observed by Goormaghtigh in $1917$ that there exist examples  with more than one solution when $n$ is odd. For instance $\frac{2^5-1}{2-1}=31=\frac{5^3-1}{5-1}$.
For a discussion and for some relatively recent results on the Diophantine equation $$\frac{x^\ell-1}{x-1}=\frac{y^k-1}{y-1},$$ we refer the reader to the work of Bugeaud and Shorey~\cite{Yann}.

\begin{lemma}\label{numb-cop}Let $n\in\mathbb{N}$, with $n\geq 3$.
\begin{itemize}
\item[i)]If $n$ is even, then the number of coprime $3$-partitions of $n$ belonging to some proper primitive group of degree $n$ is at most $\frac{\log_3(n)+1}{4}\omega(n-1)+2.$
\item[ii)]If  $n$ is  odd, then the number of coprime $3$-partitions of $n$ belonging to some proper primitive group of degree $n$ is at most $\frac{\log_2(n)+1}{2}\omega(n-1)+2$.
\end{itemize}
 \end{lemma}
\begin{proof}
Assume that there exists at least one coprime $3$-partition $\mathfrak{p}$ of $n$ belonging to some proper primitive group of degree $n$. Then, by Lemma~\ref{-2},  the possibilities for $n$ and $\mathfrak{p}$ are listed in {\rm Table~\ref{table1}} when $n$ is even and in Table~\ref{table2} when $n$ is odd.

We start with a preliminary number theoretic observation. Suppose $2^a=n=\frac{q^d-1}{q-1}$ for some positive integers $a\ge 2$, $d\ge 2$ and some prime power $q$. Then $2$ is the only prime factor of $(q^d-1)/(q-1)$ and hence $q^d-1$ has no primitive prime divisor. From a celebrated theorem of Zsigmondy, we then have $d=2$ and $n-1=q$ is prime.

\smallskip
i)  Suppose that $n$ is even, so $n\geq4$. Suppose first that $n=2^a$ and $n-1=p$, for some positive integer $a\ge 2$ and some prime number $p$. From Table~\ref{table1} and the remark above, we deduce  that $[1,1,2^a-2]$ and $[2,2^{a-1}-1,2^{a-1}-1]$ are the only possibilities for $\mathfrak{p}$.
Suppose next that $n=2^a$ and $n-1$ is not prime. From Table~\ref{table1} and the remark above, we deduce that $[2,2^{a-1}-1,2^{a-1}-1]$ is the only possibility for $\mathfrak{p}$. Similarly, if $n=22$, then $[1,7,14]$ is the only possibility for $\mathfrak{p}$; and, if $n=10$, then $[1,3,6]$ and $[1,1,8]$ are the only possibilities for $\mathfrak{p}$.

In the light of Table~\ref{table1}, it remains to consider those even integers $n$ such that
\begin{equation}\label{formq}
n=\frac{q^d-1}{q-1},
\end{equation}
 for some odd prime power $q$ and some even integer $d\ge 2$, with $n$ not a $2$-power and with $n\notin \{ 10,\,22\}$. Fix such an integer $n$. Then, recalling the definition \eqref{Q}, the cardinality $|Q(n)|$ is the number of possible representations of $n$ of the form \eqref{formq}.  By Lemma~\ref{-1}\,i),iii), the number of such representations  is, in general, at most $\omega(n-1)$ and unique when $(q,2)\in Q(n),$ for some prime power $q.$
 For each such representation, that is, for each $(q,d)\in Q(n)$, we need an upper bound on the number of $3$-partitions of the form
\begin{equation}\label{d-type}
\left[\frac{q^{d_1}-1}{q-1},\,\frac{q^{d_2}-1}{q-1},\,\frac{(q^{d_1}-1)(q^{d_2}-1)}{q-1}\right],
\end{equation}
 where $d=d_1+d_2$ and $\gcd(d_1,d_2)=1.$ Recall that each such partition is coprime by Lemma \ref{family}.

If $(q,2)\in Q(n)$, for some prime power $q$, since $d=2$,  the only partition of $n$ of type~\eqref{d-type} is  $[1,1,n-2]$. Recalling that in this case $|Q(n)|=1$ the global contribute for the coprime $3$-partitions is $1$.

Assume next that there exists no prime power $q$ such that  $(q,2)$ belongs to $ Q(n)$. Pick $(q,d)\in Q(n)$. Then $d>2$ and the number of unordered solutions  of the equation $d=d_1+d_2$, subject to $\gcd(d_1,d_2)=1$,  is equal to the number $p_2(d)'$ of coprime $2$-partitions of $d$  which, by Lemma~\ref{l:1}\,i), equals $\frac{\varphi(d)}{2}$ (since $d>2$).
Moreover, since $n$ is even, we must have $q$ odd and $d$ even. Thus,  by Lemma~\ref{phi},  $\frac{\varphi(d)}{2}\le \frac{d}{4}$ and, by Lemma~\ref{-1}\,ii),
$d < \log_q(n)+1 \leq \log_3(n)+1$.  Thus each of the at most  $\omega(n-1)$ pairs $(q,d)\in Q(n)$ contributes at most
$\frac{\log_3(n)+1}{4}$ coprime $3$-partitions.

Summarising all the possible cases we obtain that the number of coprime $3$-partitions of $n$ belonging to some proper primitive group of degree $n$ is at most $\frac{\log_3(n)+1}{4}\omega(n-1)+2.$

\smallskip

ii) Now suppose that $n\geq 3$ and $n$ is odd. The proof is similar to the the proof of i) and we omit some details.
The possibilities for $n$ and $\mathfrak{p}$ are listed in {\rm Table ~\ref{table2}}.
Suppose $n=\frac{q^d-1}{q-1}$, for some prime power $q$ and some $d\geq 2$. Arguing as in i), the number of coprime $3$-partitions of type ~\eqref{d-type} is at most
$$
\frac{\varphi(d)}{2} \omega(n-1)\leq \frac{d}{2}\, \omega(n-1)\le\frac{\log_2(n)+1}{2}\omega(n-1).
$$
%
Since a number can be both a $p$-power and of the form $n=\frac{q^d-1}{q-1},$ the number of coprime $3$-partitions of $n$ belonging to some proper primitive group is at most $ \frac{\log_2(n)+1}{2}\omega(n-1)+2$.
\end{proof}

\section{The linear bounds }\label{sec:6}
We begin by defining the functions $\zeta_2:\{n\in\mathbb{N}:n\geq 4\}\rightarrow \mathbb{Q}$  and $f :\{n\in\mathbb{N}:n\geq 4\}\rightarrow \mathbb{R}$ given, for every $n\in \mathbb{N}$ with $n\geq 4$, respectively by
\begin{align*}
\zeta_2(n)&:=\frac{1}{12}\prod_{\substack{p\mid n\\p\,\mathrm{ prime}}}\left(1-\frac{1}{p^2}\right),
&&f(n):=
\begin{cases}
\frac{1}{n^2}\frac{\log_3(n)+1}{4}\omega(n-1)+\frac{2}{n^2}+\frac{2}{n^{1/2}}&\textrm{if }n \textrm{ is even},\\
\frac{1}{n^2}\frac{\log_2(n)+1}{2}\omega(n-1)+\frac{2}{n^2}+\frac{2}{n^{1/2}}&\textrm{if }n \textrm{ is odd}.\\
\end{cases}
\end{align*}
By Lemma~\ref{l:1}\,ii),  $p_3(n)'=\zeta_2(n)n^2 $, and as an immediate consequence of the expression for $\zeta_2(n)$ and the lower bound for it given in Lemma~\ref{l:1}\,ii), we have
\begin{equation}\label{range}
\frac{1}{2\pi^2}< \zeta_2(n)< \frac{1}{12}.
\end{equation}

\begin{theorem}\label{gamma-sym-even}
Suppose that either $G=\Sym(n)$ with $n$ even, $n\ge 4$, or $G=\Alt(n)$ with $n$ odd, $n\geq 5$. Then
 $$
 \gamma(G)\geq \left\lceil\frac{n+1}{2}\left(1-\sqrt{1-8
\zeta_2(n)+8f(n)+\frac{16n+8}{(n+1)^2}(\zeta_2(n)-f(n))}\,\right)\right\rceil.
$$
\end{theorem}

\begin{proof}
Let $G\in \{\Sym(n)\mid n\geq 4 \hbox{ even}\}\cup \{\Alt(n)\mid n\geq 5 \hbox{ odd}\}$
and let $\delta$ be a minimal basic set for $G$ with maximal components.
Let $\mathcal{I}:=\{P_{x_i}\mid i\in \{1,\ldots,\ell\}\}$ be the set of intransitive components in $\delta$
 with $\ell=|\mathcal{I}|$.
Since we are dealing with the symmetric group only when the degree is even, and with the alternating group only when the degree is odd, then the primitive components in $\delta$ covering the coprime $3$-partitions of $n$ are proper primitive subgroups of $G$.

Recall that, by Lemma  \ref{l:1}\,ii), $p_3(n)'=\zeta_2(n)n^2 $. From Lemma~\ref{l:7}, the number of coprime $3$-partitions of $n$ covered by the imprimitive components in $\delta$ is at most $2n^{3/2}$ and, from Lemma~\ref{numb-cop}, the number of coprime $3$-partitions of $n$ covered by the proper primitive components of $\delta$ is at most $\frac{\log_3(n)+1}{4}\omega(n-1)+2$ when $n$ is even, and at most $\frac{\log_2(n)+1}{2}\omega(n-1)+2$ when $n$ is odd. Therefore, the remaining coprime $3$-partitions must be covered by the components in $\mathcal{I}$. Now by Lemma~\ref{l:new}, the number of coprime $3$-partitions covered by $\ell$ intransitive components is at most $\frac{\ell}{2}(n-\ell +1)$. On the other hand it is
$p_3(n)'$ minus the number covered by imprimitive or proper primitive components. Thus, in terms of the functions $\zeta_2$ and $f$
 defined above, we have
\begin{equation}\label{eq:ell}\frac{\ell}{2}(n-\ell+1)\ge \zeta_2(n)n^2-f(n)n^2,
\end{equation}
or equivalently,
\begin{equation}\label{eq:ellnew}\ell^2-(n+1)\ell+2n^2( \zeta_2(n)-f(n))\leq 0.
\end{equation}
Hence,
\begin{equation}\label{eq:ell2}\ell\ge \frac{n+1- (n+1)\sqrt{1-\frac{8n^2}{(n+1)^2}(\zeta_2(n)-f(n))}}{2}=\frac{n+1}{2}\left(1-\sqrt{1-8
\zeta_2(n)+8f(n)+\frac{16n+8}{(n+1)^2}(\zeta_2(n)-f(n))}\,\right).
\end{equation}
Observe that the radicals above are well defined real numbers  because
$1-\frac{8n^2}{(n+1)^2}(\zeta_2(n)-f(n))> 0$. Indeed, this condition is equivalent to $\zeta_2(n)-f(n)< \frac{(n+1)^2}{8n^2}$. Since $f$ is positive, by \eqref{range},  we get  $\zeta_2(n)-f(n)<\zeta_2(n)\leq \frac{1}{12}<\frac{(n+1)^2}{8n^2}$.
\end{proof}

We now derive a new lower bound for $\gamma(G)$ that allows us understand the asymptotics better.

\begin{corollary}\label{cor:nr1} Let $n\in \mathbb{N}$, with $n\geq 4.$
Suppose that either $G=\Sym(n)$ with $n$ even or $G=\Alt(n)$ with $n$ odd. Then
$$
\gamma(G)\ge
\frac{n+1}{2}\left(1-\sqrt{1-8\zeta_2(n)}\right)-\frac{n+1}{2}\sqrt{8f(n)+\frac{16n+8}{(n+1)^2}(\zeta_2(n)-f(n))}.
$$
Moreover, for $n\ge 20$,
$$
\gamma(G)\ge
\frac{n}{2}\left(1-\sqrt{1-8\zeta_2(n)}\right)-
\frac{\sqrt{17}}{2}n^{3/4}\geq  \frac{n}{2}\left(1-\sqrt{1- 4/\pi^2}\right)-\frac{\sqrt{17}}{2}n^{3/4}\ \approx\ 0.114411\, n -\frac{\sqrt{17}}{2}n^{3/4}.
$$
\end{corollary}

\begin{proof}
We derive the first expression from the inequality in Theorem~\ref{gamma-sym-even}.
Clearly, $\sqrt{a+b}\le \sqrt{a}+\sqrt{b}$ for non-negative real numbers $a$ and $b$. Thus
\begin{align*}
\sqrt{
1
-8\zeta_2(n)
+8f(n)
+\frac{16n+8}{(n+1)^2}(\zeta_2(n)-f(n))
}&\le
\sqrt{1-8\zeta_2(n)}+\sqrt{
8f(n)+\frac{16n+8}{(n+1)^2}(\zeta_2(n)-f(n))}.
\end{align*}
We now show that  this computation is meaningful because the arguments under the square roots are positive. Indeed, $1-8\zeta_2(n)\ge 0$ because $\zeta_2(n)\le 1/12$. We now consider $8f(n)+\frac{16n+8}{(n+1)^2}(\zeta_2(n)-f(n))$. Using the upper bound $\omega(n-1)\le \log_2(n-1)$, we obtain
\begin{equation}\label{positive}f(n)\le \frac{(\log_2(n)+1)\log_2(n)}{2n^2}+\frac{2}{n^2}+\frac{2}{n^{1/2}}.
\end{equation}
Let $g(n)$ be the function on the right-hand side of~\eqref{positive}.
Calculations from elementary calculus yield $g(n)<1/(2\pi^2)$, for every $n\ge 1560$.
In particular, when $n\ge 1560$, we have $8f(n)+\frac{16n+8}{(n+1)^2}(\zeta_2(n)-f(n))\ge 0$ because $f(n)>0$ and because $\zeta_2(n)>1/(2\pi^2)$. Next, when $n\leq1559$, we have computed explicitly the formula
$8f(n)+\frac{16n+8}{(n+1)^2}(\zeta_2(n)-f(n))$
and we have checked that it is positive.

Now, we show that
\begin{align*}
\frac{n}{2}\sqrt{8f(n)+\frac{16n+8}{(n+1)^2}(\zeta_2(n)-f(n))}\le \frac{\sqrt{17}}{2}n^{3/4},
\end{align*}
from which the rest of the corollary follows. This can be verified directly with the help of a computer when $20\le n\le 63$. In particular, in the rest of our argument we assume that $n\ge 64$.

By squaring both members and re-arranging the terms, this is equivalent to proving that
\begin{align*}
8f(n)+\frac{16n+8}{(n+1)^2}(\zeta_2(n)-f(n))\le \frac{17}{n^{1/2}}.
\end{align*}
As $\zeta_2(n)<1/12$ and $f(n)> 0$, it suffices to show that
\begin{align*}
8f(n)+\frac{16n+8}{12(n+1)^2}\le \frac{17}{n^{1/2}}.
\end{align*}
From~\eqref{positive}, we obtain
\begin{align*}
8f(n)&\le 8\left(\frac{(\log_2(n)+1)\omega(n-1)}{2n^2}+\frac{2}{n^2}+\frac{2}{n^{1/2}}\right)
\le
8\left(\frac{(\log_2(n)+1)\log_2(n)}{2n^2}+\frac{2}{n^2}+\frac{2}{n^{1/2}}\right)\\
&=\frac{4(\log_2(n)+1)\log_2(n)}{n^2}+\frac{16}{n^2}+\frac{16}{n^{1/2}}.
\end{align*}
Computations from calculus yield: $16/n^2<1/(3n^{1/2})$ when $n\ge 14$, $4(\log_2(n)+1)\log_2(n)/n^2<1/(3n^{1/2})$ when $n\ge 21$, $(16n+8)/12(n+1)^2<1/(3n^{1/2})$ when $n\ge 64$. In particular,
$$8f(n)+\frac{16n+8}{12(n+1)^2}\le \frac{1}{3n^{1/2}}+\frac{1}{3n^{1/2}}+\frac{16}{n^{1/2}}+\frac{1}{3n^{1/2}}=\frac{17}{n^{1/2}}.\qedhere$$
\end{proof}

We finally observe that,
$$
\frac{1-\sqrt{1 -\frac{4}{\pi^2}}} {2}>\frac{1}{\pi^2}\approx 0.101321.
$$
The constant $\frac{1}{\pi^2}$  can be obtained by using a ``first term" approach to the inclusion-exclusion proof in Lemma~\ref{l:7} and avoiding the tricky analysis carried out in Section~\ref{intersections} over the intersections of $3$-partitions with clusters. The constant $1/\pi^2$ is the constant proposed by Mar\'oti in his personal communication.

\section{Acknowledgements}
We are profoundly in debted to Attila Mar\'oti for his insights into this problem and for permitting us to use his personal communication.

\end{document}